\documentclass[english,twoside]{article}
\usepackage[T1]{fontenc}
\usepackage[latin9]{inputenc}
\usepackage{amsmath}
\usepackage{amsthm}
\usepackage{amssymb}
\usepackage{graphicx}
\usepackage{epstopdf}
\makeatletter

\providecommand{\tabularnewline}{\\}

\numberwithin{equation}{section}
\numberwithin{figure}{section}
\usepackage[natbibapa]{apacite}
\theoremstyle{plain}
\newtheorem{thm}{\protect\theoremname}
\theoremstyle{definition}
\newtheorem{defn}[thm]{\protect\definitionname}
\theoremstyle{definition}
\newtheorem{example}[thm]{\protect\examplename}
\theoremstyle{plain}
\newtheorem{lem}[thm]{\protect\lemmaname}

\usepackage{infiniteci}

\titletag{Submission to Journal of the Korean Statistical Society}
\makeatletter
\titlerunning{\@title}
\makeatother
\authorrunning{Jonas Moss}

\renewcommand{\sqrt}[1]{{(#1)^{1/2}}}

\DeclareMathOperator{\supp}{supp}

\usepackage{enumitem}
\setlist[enumerate]{label = (\roman*)}

\makeatother

\usepackage{babel}
\providecommand{\definitionname}{Definition}
\providecommand{\examplename}{Example}
\providecommand{\lemmaname}{Lemma}
\providecommand{\theoremname}{Theorem}

\begin{document}
\title{Infinite Diameter Confidence Sets in Hedges' Publication Bias Model}
\author{Jonas Moss \orcid{0000-0002-6876-6964}\\
Department of Data Science and Analytics \\
BI Norwegian Business School\emph{}\\
\emph{jonas.moss@bi.no}}
\maketitle
\begin{abstract}
Meta-analysis, the statistical analysis of results from separate studies,
is a fundamental building block of science. But the assumptions of
classical meta-analysis models are not satisfied whenever publication
bias is present, which causes inconsistent parameter estimates. Hedges'
selection function model takes publication bias into account, but
estimating and inferring with this model is tough {\color{black} for some datasets}. Using a generalized
Gleser--Hwang theorem, we show there is no confidence set of guaranteed
finite diameter for the parameters of Hedges' selection model. This
result provides a partial explanation for why inference with Hedges'
selection model is fraught with difficulties.
\end{abstract}

\section{Introduction}

A meta-analysis is a statistical analysis that quantitatively combines
results from separate scientific studies. When the studies measure
the same phenomenon, pooling of information allows us to predict the
common effect size with larger precision than we could have done with
one study alone. Meta-analyses are ubiquitous in the empirical sciences
and forms a key component of most systematic reviews such as Cochrane
reviews \citep{Higgins2019-vv}.

Most meta-analytic techniques assume honest and unbiased reporting
of results. But there is ample evidence that the scientific literature
is not unbiased, as studies with significant \textit{p}-values
tend to be published with a greater probability than other studies
\citep{Easterbrook1991-ph}, a phenomenon called publication bias
by \citet{sterling1959publication} and the file-drawer problem by
\citet{Rosenthal1979-pm}. When publication bias is present, there
is no reason to trust the results of meta-analytic methods that do
not account for it, as the parameter estimates will be inconsistent
\citep{Carter2019-rw}. 

Hedges' \citeyearpar{hedges1992modeling} publication bias model takes publication bias explicitly into account using a selection model, and is arguably the most appropriate model for publication bias \citep{Carter2019-rw}. Despite there being an $\mathtt{R}$ \citep{Team2013-tt} package for maximum likelihood estimation of this model, called $\texttt{weightr}$ \citep{Coburn2019-ec}, the model has not yet taken off. Its maximum likelihood estimation methods are numerically unstable and its estimates can be off even when they converge \citep{Coburn2019-ec,Stanley2005-ng}. The estimate of the mean effect size may be negative and of unrealistically large magnitude, and the estimated heterogeneity parameter might be improbably large. It turns out there are ridges in the likelihood function that can be linked to this behavior \citep{McShane2016-rb}, but it has not been stated in clear terms exactly what the consequences are for inferential procedures. The purpose of this note is to explain why Hedges' publication bias performs poorly, by showing there is no confidence set for the mean effect size that has infinite diameter with probability zero. 

\section{Hedges' publication bias model}

The most popular and well-known meta-analysis method is the random effects model with normal likelihoods \citep{hedges1998fixed}. Written in hierarchical notation, it equals
\begin{eqnarray*}
\theta_{i} & \sim & N(\theta_{0},\tau),\\
x_{i}\mid\theta_{i},\sigma_i & \sim & N(\theta_{i},\sigma_{i}).
\end{eqnarray*}
Here $x_{i}$ is the effect size and $\sigma_{i}$ is the standard deviation of the $i$th study,  $i=1,\ldots,N$. Following the convention in meta-analysis, we assume all $\sigma_{i}$s to be known. The mean parameter $\theta_0$ is the population effect size, $N(\theta_{0},\tau)$ is the effect size distribution, and $\tau$ is the heterogeneity parameter. The purpose of the effect size distribution is to model the fact that most effect size estimates plugged into a meta-analysis do not appear to measure the same phenomenon. By integrating out $\theta_{i}$, we find the density of $x_{i}$,
\[
f(x_{i};\theta_{0},\tau,\sigma_{i})=\phi(x_{i};\theta_{0},\sqrt{\tau^{2}+\sigma_{i}^{2}}),
\]
where $\phi$ is the density of a normal random variable. 

We will assume that the random effects meta-analysis model is true in the absence of publication bias. The mechanisms that cause publication bias modify the density in a suitable way. Consider the case when only significant studies at some specified level $\alpha$ are published. Assuming one-sided tests, the \textit{p}-values are $u_{i}=\Phi(-x_{i}/\sigma_{i})$, or normal one-sided \textit{p}-values. We will only deal with with one-sided \textit{p}-values in this paper, as there is usually, but not always, just one direction that is interesting to researchers, reviewers, and editors. A one-sided \textit{p}-value can also be used if the researchers reported a two-sided value, since $p=0.05$ for a two-sided hypothesis corresponds to $p=0.025$ for a one-sided hypothesis, et cetera.

Define $c_{\alpha}=\Phi^{-1}(1-\alpha)$, the cutoff for significance at level $\alpha$. The \emph{basic publication bias model} is a truncated normal model with density
\begin{equation}
f(x_{i};\theta_{0},\sigma_{i})=\Phi\left(\frac{\theta_{0}-c_{\alpha}}{\sqrt{\sigma_{i}^{2}+\tau^{2}}}\right)^{-1}\phi(x;\theta_{0},\sqrt{\sigma_{i}^{2}+\tau^{2}})1(x_{i}/\sigma_{i}>c_{\alpha}],\label{eq:selection for significance model}
\end{equation}
where $1[A]$ is the characteristic function of $A$. This model for publication bias was introduced by \citet{hedges1984estimation} in the context of $F$-distributions.

The basic publication bias model is unrealistic. It requires that no non-significant studies are published. But even in the fields most severely affected by publication bias, such as psychology, a non-negligible number of non-significant studies are published \citep{Motyl2017-dx}. Moreover, the basic publication bias model does not allow for different cutoffs for significance. It is likely that some editors will accept studies reaching a significance at $\alpha=0.025$, corresponding to $x_{i}/\sigma_{i}>1.96$ but not at $\alpha=0.05$, corresponding to $x_{i}/\sigma_{i}>1.64$.

These problems can be rectified by adopting the selection model for publication bias of \citet{iyengar1988selection}, which models the following scenario.
\begin{quote}
\textbf{Publication bias scenario.} Alice the editor receives a study with the \textit{p}-value $u$. Her publication decision is a random function of this \textit{p}-value. That is, she will publish the study with some probability $w(u)$ and reject it with probability $1-w(u)$. Every study you will ever read in Alice's journal has survived this selection mechanism, the rest are lost forever.
\end{quote}
Let $w(u_{i})$ be a function of the \textit{p}-value $u_{i}=\Phi(-x_{i}/\sigma_{i})$ taking values in $[0,1]$. Then $w(u_{i})$ is a probability for every $u_{i}$, and the selection model
\begin{equation}
f(x_{i};\theta_{0},\sqrt{\sigma_{i}^{2}+\tau^{2}})\propto\phi(x_{i};\theta_{0},\sqrt{\tau^{2}+\sigma_{i}^{2}})w(u)\label{eq:iyengar-greenhouse model}
\end{equation}
models the publication bias scenario exactly. This model can be viewed as a rejection sampling procedure \citep{flury1990acceptance,von1951various}, where $\phi$ serves as proposal distribution for $f$. Variants of this model, with and without covariates, has been studied by e.g. \citet{Dear1992-gw,Vevea1995-on,Vevea2005-xp,citkowicz2017parsimonious}.

\citet{hedges1992modeling} studies the selection model when $w$ is a step function with fixed steps. Let $\alpha$ be a vector with elements $0=\alpha_{0}<\alpha_{1}<\ldots<\alpha_{K}=1$ and $\rho$ be a $K$-ary non-negative, non-increasing vector having its first element equal to $\rho_{1}=1$ for identifiability. Define the step function $w$ based on $\alpha$ and $\rho$ as
\begin{equation}
w(u;\rho,\alpha)=\sum_{k=1}^{K}\rho^{k}1_{(\alpha_{k-1},\alpha_{k}]}(u).\label{eq:step function}
\end{equation}
We call the selection model with a step function \emph{Hedges' publication bias model}. Its density is
\begin{equation}
f(x_{i};\theta_{0},\sqrt{\tau^{2}+\sigma_{i}^{2}})\propto\sum_{k=1}^{K}\rho^{k}\phi(x_{i};\theta_{0},\sqrt{\tau^{2}+\sigma_{i}^{2}})1{(\alpha_{k-1},\alpha_{k}]}(u_{i}).\label{eq:hedges model}
\end{equation}
Interpreting Hedges' publication bias model is easy. When the editor receives a study with \textit{p}-value $u$, she finds the $k$ such that $u\in(\alpha_{k-1},\alpha_{k}]$ and accepts with probability $\rho^{k}$. Since $\rho^{1}=1$, she always accepts when $u\in[0,\alpha_{1}]$. The vector $\rho$ is non-increasing since a publication decision based solely on \textit{p}-values should always act favorably towards lower \textit{p}-values. The parameters $(\theta_{0},\tau,\rho)$ of the model are identified when $\alpha$ is fixed \citep[Web Appendix A]{moss2019modelling}. 

{\color{black}It is probably not possible to generalize the results of this paper to selection models that do no follow the step function model. Lemma \ref{lem:generalized shifted exponential}, about the truncated normal, is crucial in the proof of our main result, Theorem \ref{thm:general publication bias}. Truncated densities only appear in step function models, not models with continuous selection functions, such as the Probit selection function of \cite{Copas2013-qi} or the one-parameter selection functions of \cite{Preston2004-pv}.}

Hedges' publication bias model allows both for non-significant studies to be published and allows the editor to act differently towards different cutoffs such as $\alpha=0.025$ and $\alpha=0.05$. In addition, the model can approximate any non-increasing selection function $w$ by increasing the number of steps. {\color{black}In applications, the parameters $\mu$, $\tau$, and $\rho$ are estimated from the data, while $\alpha$ is fixed by the researcher, for instance at $\alpha = (0.025, 0.05, 1)$. We recommend using these cutoffs, as it is well known that applied journals frequently demand statistical significance at this level. Since both two-sided and one-sided \emph{p}-values occur, we need to include $0.025$ in addition to $0.05$.}

We can write Hedges' model as a mixture model on the form
\begin{equation}
f(x_{i};\theta_{0},\sqrt{\tau^{2}+\sigma_{i}^{2}})=\sum_{k=1}^{K}\pi^{k}f^{k}(x_{i};\theta_{0},\sqrt{\tau^{2}+\sigma_{i}^{2}}).\label{eq:mixture model formulation}
\end{equation}
Where $f^{k},k\leq K$ are normal densities
truncated to $(\Phi^{-1}(1-\alpha_{k-1}),\Phi^{-1}(1-\alpha_k)]$, and $\pi^{k}$ are mixture probabilities, i.e., $\pi^{k}>0$ for each $k$ and $\sum_{k=1}^{K}\pi^{k}=1$. The mixture probabilities $\pi^{k}$ are functions of $(\theta_{0},\tau,\sigma_{i},\rho)$, see the appendix (p. \pageref{eq:pi_i formula}) for their formula.

The main benefit of Hedges' publication bias model (\ref{eq:hedges model}) is how it models \textit{p}-values based publication bias directly, there is no approximation involved. If you believe in the random effects meta-analysis model and the \textit{p}-value based publication bias scenario, Hedges' publication bias model is simply the correct model. Most statistical methods correcting for publication bias in the literature either do not make use of an explicit statistical model or do not estimate the parameters $\theta_{0}$ and $\tau$. For instance, the funnel plot of \citet{Egger1997-rd} is a graphical method, while the trim-and-fill method of \citet{Duval2000-ct} is a non-parametric method based on making the funnel plot symmetric. \citet{Stanley2005-ng,stanley2014meta} discuss various misspecified regression-based estimators of the corrected effect size $\theta$ based on the fixed effect Hedges' publication bias model. The estimating \textit{p}-curve method of \citet{Simonsohn2014-cn} and \textit{p}-uniform of \citet{Van_Assen2015-qs,Van_Aert2016-cu} are two methods for dealing with publication bias hailing from psychology. Both are based on a variant of the basic publication bias model, but with fixed instead of random effects, and both employ somewhat unusual estimation methods \citep{McShane2016-rb}. Since there is ample evidence of heterogeneity in meta-analysis, restricting oneself to the fixed effects meta-analysis is a mistake.

Hedges' model has some downsides. It models only bias due to selection
of \textit{p}-values, not every source of bias, such as language bias
\citep{Egger1998-kj}. Second, it may not be best model for biases
with other causes than the publication process, such as \textit{p}-hacking
\citep{simmons2011false}. \citet{moss2019modelling} propose a related
model that may be more successful at correcting for \textit{p}-hacking. 

{\color{black} Hedges' model can be hard to estimate, especially if the data is unfavourable}. \citet[Section 6.3]{Stanley2005-ng}
discusses three problematic cases in economics when maximum likelihood
was used to estimate Hedges' publication bias model. \citet[Appendix A]{McShane2016-rb}
notes that while estimation of the basic publication bias model is
hard, introducing the heterogeneity parameter exacerbates the problem.
The likelihood function has contours following approximately $\tau\propto|\theta|^{1/2}$.
\begin{figure}[t]
\noindent \begin{centering}
\includegraphics[scale = 0.5]{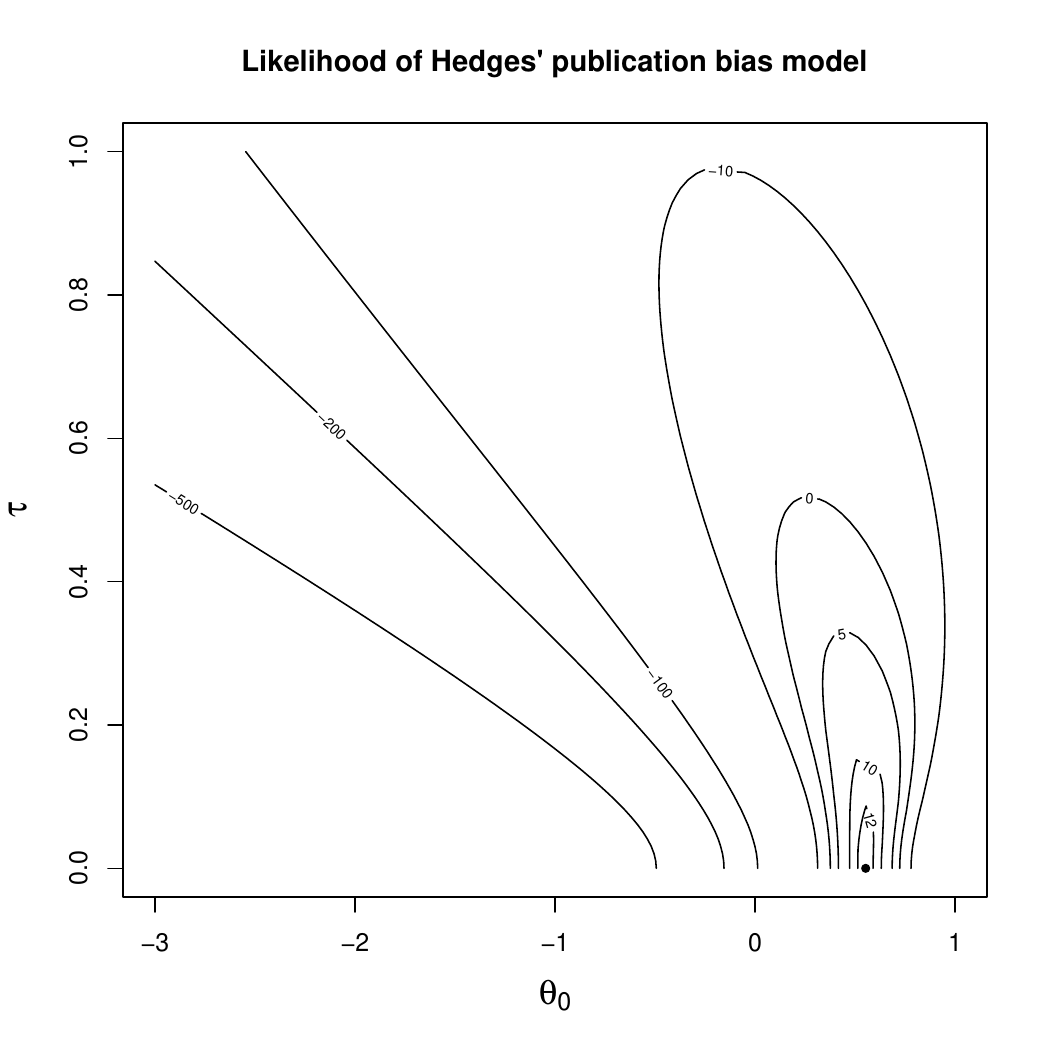}
\par\end{centering}
\caption{\label{fig:contour lines}Contour lines for the log-likelihood for
the simple publication model using power posing data of \citet{Cuddy2018-kp}. }
\end{figure}
Figure \ref{fig:contour lines} shows the contour lines for the meta-analysis of \citet{Cuddy2018-kp} where the selection probabilities of the step function (\ref{eq:step function}) are fixed at $\rho=(1,0.6,0.1)$ and $\alpha = (0, 0.025, 0.05, 1)$. Only just around the maximum at $\hat{\theta}_{0}=0.55$ and $\hat{\sigma}=5\cdot10^{-7}$
can the likelihood be approximated with a quadratic function.

{\color{black} Estimation of Hedges' model is especially hard when almost all observations lie close to the the cutoffs. In our experience, estimation works well when the data is sufficiently well spread, for instance when some observations are far away from the cutoffs and there are observations that have failed to reach significance. As a simple example of a case when estimation fails to work, consider the observation vector $x = (1.96, 1.96, 1.96, 1.96, 1.96, 1.64, 1.64, 1.64, 1.64, 1.64, -1)$ when the standard deviation of each observation is $\sigma_i=1$. We employ $\alpha = (0, 0.025, 0.05, 1)$, which implies cutoffs at $\Phi^{-1}(0.975)$ and $\Phi^{-1}(0.95)$. In this case, $10$ out of $11$ observations are very close to the cutoffs. When we run the Hedges selection model on this data, using the $\mathtt{sel}$ function of the $\mathtt{R}$ package $\mathtt{metafor}$ \citet{Viechtbauer2010-cf}, the model does not converge. If the last observation is changed to $0$ instead of $-1$, the model converges, but the Hessian cannot be inverted, and the parameter estimate for $\theta_0$ equals $-1.8$. Situations similar to this one are most common when $n$ is small.}

\section{Confidence sets of infinite diameter}

Fix some measurable space $(\Omega,\mathcal{F})$, let $\mathcal{P}$
be a family of dominated probability measures defined on this measurable
space, and $\Pi$ a partition of $\mathcal{P}$. Recall that a partition
of $\mathcal{P}$ is a collection of disjoint non-empty subsets $\pi$
of $\mathcal{P}$ such that $\bigcup_{\pi\in\Pi}\pi=\mathcal{P}$.
When $p$ is a density associated with a $P\in\mathcal{P}$, we will
use the standard notation $[p]$ to denote the unique part $\pi$ containing
$P$. {\color{black} Instead of partitions, we could have used a formulation with main parameters $\theta$ and nuisance parameters $\eta$, and defined the rejection set for $\theta$ as $\sup_{\eta}\sup_{\theta}P_{\theta,\eta}(R(\theta))\leq\alpha$. We have decided to use partitions for two reasons: First, there are no unambiguous nuisance and main parameters in our applications, making notation using nuisance parameters confusing. Second, the upcoming Theorem \ref{thm:Gleser--Hwang} can be applied in purely non-parametric situations, where the mention of nuisance and main parameters is even more confusing.} 
\begin{defn}
\label{def:confidence set} A \emph{confidence set} of level $\alpha$
is a family of rejection sets $\{R(\pi)\},\pi\in\Pi$ such that
\[
\sup_{\pi\in\Pi}\sup_{P\in\pi}P(R(\pi))\leq\alpha.
\]
If the inequality is an equality, the confidence set has \emph{size}
$\alpha$. 
\end{defn}

This definition of confidence sets might look slightly unfamiliar,
but it is a straight-forward generalization of the definitions in
\citet[Definition 9.1.5]{Casella2002-lg} and \citet[Section 3.5]{lehmann2006testing}.
Usually, a confidence is defined as a set $C$ adhering to the relation
\begin{equation}
\pi_{0}\in C\iff\omega\notin R(\pi_{0}).\label{eq:c-confidence set}
\end{equation}
That is, $\pi_{0}\in C$ if and only if we accept the null-hypothesis
$H_{0}:\pi=\pi_{0}$, but we will only need the formulation using rejection
sets in this paper. When confidence sets are defined in terms of rejection
sets, there is sometimes no partition $\Pi$ to take the supremum
over, and the definition reduces to $\sup_{P\in\mathcal{P}}P(R(P))\leq\alpha$.
The term\emph{ confidence interval} is far more common than confidence
set, but this requires that the $C$ in equation \ref{eq:c-confidence set}
is an interval, which we will not require here.

The following example should make Definition \ref{def:confidence set}
clearer. 
\begin{example}
\label{exa:t-test} Consider the usual \emph{t}-confidence interval
with $n$ observations. In this case, $\mathcal{P}$ contains all measures
$P_{\mu,\sigma}^{n}$, where $P_{\mu,\sigma}$ is the probability
measure of a normal with mean $\mu$ and standard deviation $\sigma$.
Here $Q^{n}$ denotes the $n$-fold product measure of $Q$, corresponding to $n$ independent samples from $Q$ when $Q$ is a probability measure.  {\color{black} The \emph{t}-confidence interval is an exact confidence interval for $\mu$ no matter what $\sigma>0$, the nuisance parameter, is. Since the test is exact, rejection sets $R(\mu)$ satisfy $P^n_{\mu,\sigma}(R(\mu))=\alpha$ for all $\sigma,\mu$.} We can formulate the confidence set in terms of partitions too. Let $\pi(\mu)=\{P_{\mu,\sigma}^{n}\mid\sigma>0\}$
contain all normal probability measures with mean $\mu$ and some
positive standard deviation. Then $\{\pi(\mu)\},\mu\in\mathbb{R}$
defines a partition of $\mathcal{P}$. Then the two-sided \emph{t}-confidence
interval is a confidence set of size $\alpha$ with partition {\color{black}$\Pi = \{\pi(\mu)\},\mu\in\mathbb{R}$}
according to Definition \ref{def:confidence set}.
\end{example}

{\color{black} Now we must find out how to measure the size of confidence sets. To make our results as general as possible, we will let the \emph{size function} be any non-negative function $\left\Vert \cdot\right\Vert :\Pi\to[0,\infty)$. In most cases, the size function will be a norm, but any non-negative function is a valid size function. For instance, in the \emph{t}-confidence interval example above, $\left\Vert \cdot\right\Vert $
can be taken to be $\left\Vert \pi\right\Vert =|\mu|$ for the unique
$\mu$ associated with each $\pi$. }

The diameter of a confidence set is the random variable
\begin{equation}
D(\omega)=\sup_{\pi\in\Pi}\{\left\Vert \pi\right\Vert \mid\omega\notin R(\pi)\}.\label{eq:diameter}
\end{equation}
The diameter tells you the size of the largest accepted $\pi$.
We will assume that $D$ is Borel measurable.
\begin{defn}
\label{def:infinite diameter}A confidence set has infinite diameter
with $P$-positive probability if $P(D=\infty)>0$. It has infinite
diameter with positive probability if $P(D=\infty)>0$ for all $P\in\mathcal{P}$.
\end{defn}

{\color{black} The original Gleser--Hwang theorem \citep[Theorem 1]{gleser1987nonexistence} is defined for pairs of parameters $\theta_1,\theta_2$, where $\theta_2$ is a nuisance parameter and the confidence set is constructed for a functional $\gamma(\theta_1)$. The following generalization does not require any nuisance parameters. Using partitions, it can be used both with and without nuisance parameters, as well as in non-parametric settings.} Its proof is in the appendix (p. \pageref{proof:Gleser--Hwang}). 
\begin{thm}[Gleser--Hwang theorem]
\label{thm:Gleser--Hwang} Suppose there is a sequence $\{p_{n}\}$
of densities derived from $\mathcal{P}$ satisfying the following: 
\begin{enumerate}
\item There is a density $p^{\star}$ such that $p_{n}$ converges to $p^{\star}$
pointwise,
\item $\supp p\supseteq\supp p^{\star}$ for all densities $p$ derived
from $\mathcal{P}$,
\item the size of the equivalence class $[p_{n}]$ goes to infinity as $n$
increases, $\left\Vert [p_{n}]\right\Vert \to\infty$.
\end{enumerate}
Then every confidence set with level $\alpha>0$ has infinite diameter
with positive probability.
\end{thm}

Following the terminology of \citet{berger1999integrated}, we will
say that families $\mathcal{P}$ of probabilities satisfying the conclusion
of Theorem \ref{thm:Gleser--Hwang} for a suitable partition $\Pi$
belong to the Gleser--Hwang class. To make the Gleser--Hwang class
more familiar, we will present two examples. More examples can be
found in the papers of \citet{gleser1987nonexistence} and \citet{berger1999integrated}.

Fieller's problem is the best known case of a badly behaved confidence
set. Let $(X,Y)$ be an observation from a bivariate normal $N([\mu_{1},\mu_{2}],I\sigma^{2})$,
where $\sigma^{2}$ is known. We want to form a confidence set for
the ratio $\mu_{2}/\mu_{1}$. The most famous confidence set is due
to \citet{Fieller1940-lg}. His confidence set can be finite, the
whole real line, or the union of two disjoint semi-infinite intervals, all with positive probability \citep{Koschat1987-dk}.

But it is not only Fieller's confidence set that might be infinitely
long. The Gleser--Hwang theorem can be used to show that \textit{every} confidence
set for $E(Y)/E(X)$ must be infinitely long with positive probability.
This result is almost independent of the distribution of $X$ and
$Y$. To state this result in our notation, let $\mathcal{P}$ be
a family of bivariate distributions over $(X,Y)$. All of these distributions
have the same support, and all of them have finite means $E_{p}(X)$
and $E_{p}(Y)$. Moreover, assume $E_{p}(X)=0$ is attainable for
some $p\in\mathcal{P}$. Define the partition $\Pi$ by $p,q\in\pi$
if and only if $E_{p}(X)/E_{p}(Y)=E_{q}(X)/E_{q}(Y)$, and let $\left\Vert [p]\right\Vert =|E_{q}(X)/E_{q}(Y)|$,
i.e., the ratio of means. Choose a sequence $p_{n}(x,y)=p(x,y)$,
where $p(x,y)$ is density with means $E_{p}X>0$ and $E_{p}Y=0$.
Then $\left\Vert [p]\right\Vert =\infty$, the conditions of Theorem
\ref{thm:Gleser--Hwang} are satisfied, and every confidence set with
level $\alpha>0$ has infinite diameter with positive probability.

Another example is due to \citet{bahadur1956nonexistence}, who studies non-parametric testing of
the mean, and concludes the mean cannot be meaningfully tested. They
are working with a family $\mathcal{P}$ of densities over $\mathbb{R}$
that covers all finite means, has finite variances, and is closed
under convex combinations. Similar problems were considered by \citet{romano2004non}
and \citet{Donoho1988-hg}.

Using the Gleser--Hwang theorem, it is easy to verify that every confidence
set has infinite diameter with positive probability. Define the partition
$\Pi$ by $p,q\in\pi$ if and only if $E_{p}(X)=E_{q}(X)$, and let
$\left\Vert [p]\right\Vert =|E_{p}(X)|$. Let
\[
p_{n}(x)=\left(1-\frac{1}{n}\right)q_{0}(x)+\frac{1}{n}q_{n^{2}}(x),
\]
where $q_{0}$ has mean $0$ and $q_{n^{2}}$ has mean $n^{2}$. Then
$\left\Vert [p_{n}]\right\Vert =n$, the conditions of Theorem \ref{thm:Gleser--Hwang}
are satisfied, and every confidence set with level $\alpha>0$ has
infinite diameter with positive probability.

There are several natural candidates for $\Pi$ when working with
Hedges' selection function model (\ref{eq:hedges model}). We will
work with three of them. First, consider the partition where $p,q\in\pi$
if and only if $p,q$ have the same mean effect size parameter $\theta_{0}$.
We will equip this partition with the size function $\left\Vert \pi\right\Vert =|\theta_{0}|$,
and it corresponds to a confidence set for $\theta_{0}$. Second,
consider the partition where all $p,q\in\pi$ have the same heterogeneity
parameter $\tau$, equipped with $\left\Vert \pi\right\Vert =\tau$.
Finally, we will work with the partition where all $p,q\in\pi$ have
the same heterogeneity parameter $\tau$ and population effect size
$\theta_{0}$, and equip it with $\left\Vert \pi\right\Vert =\sqrt{\theta_{0}^{2}+\tau^{2}}$.
This information is summarized in Table \ref{tab:Possible partitions}
for convenience. 

\begin{table}
\noindent \centering{}\caption{\label{tab:Possible partitions}The three partitions $\Pi$ for the
selection function model}
{\small{}}%
\begin{tabular}{llll}
 & {\small{}Symbol} & {\small{}Size $\left\Vert \cdot\right\Vert $} & {\small{}Confidence set}\tabularnewline
\hline 
{\small{}Mean effect size} & {\small{}$\theta_{0}$} & {\small{}$\left\Vert \pi\right\Vert =|\theta_{0}|$} & {\small{}Confidence set for $\theta_{0}$}\tabularnewline
{\small{}Heterogeneity parameter} & {\small{}$\tau$} & {\small{}$\left\Vert \pi\right\Vert =\tau$} & {\small{}Confidence set for $\tau$}\tabularnewline
Both parameters & {\small{}$(\theta_{0},\tau)$} & {\small{}$\left\Vert \pi\right\Vert =\sqrt{\theta_{0}^{2}+\tau^{2}}$} & {\small{}Joint confidence set for $(\theta,\tau)$}\tabularnewline
\end{tabular}
\end{table}

Let us take a look at the basic publication bias model (\ref{eq:selection for significance model})
again. To use Theorem \ref{thm:Gleser--Hwang} we need a witnessing
sequence of functions $p_{n}\to p$ satisfying the conditions $\text{(ii)}$
and $\text{(iii)}$. The next lemma shows how to make such a witness
for the truncated normal. Its proof is in the appendix (p. \pageref{proof:generalized shifted exponential}).
\begin{lem}
\label{lem:generalized shifted exponential} Let $f_{n}$ be a normal
density truncated to $[a,b)$, where $b=\infty$ is allowed, with
underlying mean $\theta_{n}=-n$ and standard deviation $\sigma_{n}^{2}=n+c$
for some $c\in\mathbb{R}$. Then $f_{n}$ converges pointwise to $\exp(-x)/[\exp(-a)-\exp(-b)]$,
the distribution of an exponential variable truncated to $[a,b)$.
\end{lem}

Using Lemma \ref{lem:generalized shifted exponential} it is not hard
to show that the basic publication bias model (\ref{eq:selection for significance model})
is a member of the Gleser--Hwang class.
\begin{thm}
\label{thm:simple publication bias} Assume we have $N$ independent samples
from the basic publication bias model (\ref{eq:selection for significance model}). Then any confidence set for $\theta_{0},\tau$, or
$(\theta_{0},\tau)$ with level $\alpha>0$ will have infinite diameter
with positive probability. 
\end{thm}

\begin{proof}
\label{proof:simple publication bias}Let $\Pi$ be the partition
of $\mathcal{P}$ where $p,q\in\pi$ if and only if they share the
same $\theta_{0}$, and let $||[p]||=|\theta_{0}|$. We are dealing
with products of densities of the form (\ref{eq:selection for significance model}),
that is,
\[
p(x)=\prod_{i=1}^{N}\Phi\left(\frac{\theta_{0}-c_{\alpha}}{\sqrt{\sigma_{i}+\tau}}\right)^{-1}\phi(x_{i};\theta_{0},\sqrt{\sigma_{i}^{2}+\tau^{2}}),
\]
where $\sigma_{i}$ are known parameters. From Lemma \ref{lem:generalized shifted exponential},
$p_{n}$ converges to a product of truncated exponentials when $\theta_{n}=-n$
and $\tau_{n}^{2}=n$. Since $||[p_{n}]||=n$, the three conditions
of Theorem \ref{thm:Gleser--Hwang} are satisfied. The proofs for
$||[p]||=\tau$ and $||[p]||=\sqrt{\tau^{2}+\theta_{0}^{2}}$ are
similar and omitted.
\end{proof}
Proving the analogue of Theorem \ref{thm:simple publication bias}
for Hedges' publication is only somewhat more involved. We will use
the mixture representation of (\ref{eq:mixture model formulation})
and a lemma generalizing Theorem \ref{thm:Gleser--Hwang} to a certain
kind of mixtures. 

Let $f^{1},f^{2},\ldots,f^{K}$ be a sequence of densities, $\pi=(\pi^{1},\pi^{2},\ldots,\pi^{K})$
be a probability vector, and $p=\sum_{k\leq K}\pi^{k}f^{k}$ be a mixture distribution. We will assume that the size of $[p]$ equals the size of any of its mixture components $[f^{k}]$ for some size $\left\Vert \cdot\right\Vert $,
i.e., $\left\Vert [p]\right\Vert =\left\Vert [f^{k}]\right\Vert $
for all $k$. Why we do this will be clear in the proof of Theorem
\ref{thm:general publication bias}, but think of it this way: If
all of $p$s mixture components have the same mean, the mean of $p$
equals the mean of any $f^{k}$. 
\begin{lem}
\label{prop:Mixture model corollary} Let $\mathcal{P}$ be a class
of $K$-ary mixture distributions and $||[p]||$ be as assumed above. Assume there is a sequence $p_{n}=\sum_{k\leq K}\pi_{n}^{k}f_{n}^{k}$
and a subset $K'$ such that 
\begin{enumerate}
\item For all $k\in K'$, there is a density $f^{k\star}$ such that $f_{n}^{k}$
converges to $f^{k\star}$ pointwise.
\item For all mixtures $p$, $\supp p\supseteq\supp f^{k\star}$ for all
$k\in K'$.
\item For all $k\in K'$, the size of $[f_{n}^{k}]$ goes to infinity, $\left\Vert [f_{n}^{k}]\right\Vert \to\infty$.
\item The density concentrates on the components indexed by $K'$, $\lim_{n\to\infty}\sum_{k\in K'}\pi_{n}^{k}=1$.
\end{enumerate}
Then every confidence set with level $\alpha>0$ has infinite diameter
with positive probability.
\end{lem}

\begin{proof}
We employ Theorem \ref{thm:Gleser--Hwang}. By (i) and (iv), $p_{n}$
converges pointwise to the density $$\sum_{k\in K'}\pi_{n}^{k\star}f_{n}^{k\star}=p^{\star}.$$
That $\supp p\supseteq p^{\star}$ follows from (ii) and (iv). Finally,
from the assumption that$\left\Vert [p_{n}]\right\Vert =\left\Vert [f_{n}^{k}]\right\Vert $,
we get that $\left\Vert [p_{n}]\right\Vert \to\infty$ too. 
\end{proof}
\begin{thm}
\label{thm:general publication bias} Assume we have $N$ independent samples from the publication bias model (\ref{eq:selection for significance model}), where the selection probability $\rho$ is unknown and $\alpha$ is known. Then any confidence set for $\theta_{0},\tau$, or $(\theta_{0},\tau)$ will have infinite diameter with positive probability.
\end{thm}

\begin{proof}
\label{proof:general publication bias} Let $n=1$ and consider confidence sets for $\theta_{0}$. Let $\Pi$ be the partition of $\mathcal{P}$ where $p,q\in\pi$ if and only if they share the same $\theta_{0}$, and let $||[p]||=|\theta_{0}|$. Then $||f^{k}||=|\theta_{0}|$ from the mixture representation (\ref{eq:mixture model formulation}). Using Lemma \ref{lem:generalized shifted exponential}, we see that $f^{k},k<K$ converges pointwise to truncated exponentials when $\theta_{0}=-n$ and $\tau_n^{2} = n$, so that (i), (iii) of Proposition \ref{prop:Mixture model corollary} are satisfied with the set $K'=\{1,2,\ldots, K-1\}$. Moreover, since we assume that $\rho$ is decreasing, (ii) is satisfied as well. The mixture probabilities for $k \neq K$ can be fixed at e.g. $\pi=1/(K-1)$, and (iv) is satisfied as well. The proofs for $||[p]||=\tau$ and $||[p]||=\sqrt{\tau^{2}+\theta_{0}^{2}}$ are similar and omitted. 

When $N>1$, expand the expression $\prod_{i=1}^{N}\sum_{k<K}\pi_{i}^{k}(\sigma_{i})f_{i}^{k}(\sigma_{i}),$ and use the same reasoning as in the first part of this proof.
\end{proof}

\section{Remarks}

Well-behaved confidence sets for Hedges publication bias model do
not exist, but well-behaved credibility sets do. Bayesian estimation
of Hedges' model can be made routine, as it is easy to find uncontroversial
priors for $\theta_{0}$ and $\tau$. In practical meta-analyses we
know that $\theta_{0}$ cannot be large, and is likely to be close
to $0$. Moreover, since it is common effects in meta-analyses to
be interpreted as the aggregation of many small effects, the central
limit theorem justifies using a normal prior. As we want to remove
prior mass from negative $\theta_{0}$s of large magnitude, $N(0,1)$
is a decent standard prior. Similarly, a half-normal is a reasonable
prior for the heterogeneity parameter $\tau$. \citet{moss2019modelling}
employed these priors on several examples. 

\section*{Appendix}

The following sandwich convergence theorem is used in the proof of
the Gleser--Hwang theorem. 
\begin{lem}[{\citet[Exercise 16.4(a)]{billingsley1995probability}}]
\label{lem:Dominated covergence theorem} Suppose the functions $a_{n},b_{n},f_{n}$
converge pointwise to $a,b,f$ and $a_{n}\leq f_{n}\leq b_{n}$ for all $n$.
If $\int a_{n}d\mu\to\int ad\mu$ and $\int b_{n}d\mu\to\int bd\mu$,
then $\int f_{n}d\mu\to\int fd\mu$ for any measure $\mu$.
\end{lem}

The proof of Theorem \ref{thm:Gleser--Hwang} closely follows the
proof of \citet[Theorem 1]{gleser1987nonexistence}.
\begin{proof}[Proof of Theorem \ref{thm:Gleser--Hwang}]
\label{proof:Gleser--Hwang}We can assume without loss of generality
that $\left\Vert [p_{n}]\right\Vert \geq n$, as we can choose a suitable sub-sequence
if we have to. By definition of the diameter $D$ (\ref{eq:diameter})
we see that
\[
\{D\geq n\}=\{\omega\in\Omega\mid\textrm{there is a }\pi\textrm{ such that}\left\Vert \pi\right\Vert \geq n\textrm{ and }\omega\in R^{c}(\pi)\}.
\]
It follows that, if $||[p_{n}]||\geq n$, then $R^{c}([p_{n}])\subseteq\{D\geq n\}$.
Since we assume that $\left\Vert [p_{n}]\right\Vert \geq n$ and 
\[
1-\alpha\leq P_{n}(R^{c}([p_{n}]))=\int_{R^{c}([p_{n}])}p_{n}d\mu
\]
by definition of a confidence set, we have that
\begin{equation}
0<1-\alpha\leq\int_{R^{c}([p_{n}])}p_{n}d\mu\leq\int_{D\geq n}p_{n}d\mu\label{eq:inequality with R^c}
\end{equation}
for all $n$. Since $p_{n}$ and $p^{\star}$ are densities,
\[
\lim_{n\to\infty}\int p_{n}d\mu=1=\int p^{\star}d\mu=\int\lim_{n\to\infty}p_{n}d\mu.
\]
This allows us to use Lemma \ref{lem:Dominated covergence theorem}
with $a_{n}=0$, $b_{n}=p_{n}$, and $f_{n}=1_{D\geq n}p_{n}$ and
conclude that

\begin{equation}
\int_{D\geq n}p_{n}d\mu\to\int_{D=\infty}p^{\star}d\mu.\label{eq:limit of pn}
\end{equation}
Combining equations (\ref{eq:inequality with R^c}) and (\ref{eq:limit of pn}),
we get
\[
0<1-\alpha\leq\int_{D=\infty}p^{\star}d\mu.
\]
Let $P\in\mathcal{P}$ be arbitrary, $p$ be its density, and consider
\[
P(D=\infty)=\int_{D=\infty}pd\mu\geq\int_{D=\infty\cap\supp p^{\star}}\left(\frac{p}{p^{\star}}\right)p^{\star}d\mu.
\]
Since $\int_{D=\infty}p^{\star}d\mu>0$ and $p/p^{\star}>0$ on $\supp p^{\star}$
(since $\supp p\supseteq\supp p^{\star}$ by assumption), we see that
$\int_{D=\infty\cap\supp p^{\star}}(p/p^{\star})p^{\star}d\mu>0$
too. It follows that $P(D=\infty)>0$, and, since $P$ is arbitrary,
$D$ has infinite diameter with positive probability.
\end{proof}
Now we prove Lemma \ref{lem:generalized shifted exponential}.
\begin{proof}[Proof of Lemma (\ref{lem:generalized shifted exponential})]
\label{proof:generalized shifted exponential}Let $n>-c$, so
that $\sigma_{n}^{2}>0$. Recall the well-known formula for the normal
truncated to $[a,b]$, and substitute $\theta_n = -n$ and $\sigma^2_n = n + c$,
\begin{eqnarray}
f_{n}(x) & = & \frac{1}{\Phi\left(\frac{b-\theta_{n}}{\sigma_{n}}\right)-\Phi\left(\frac{a-\theta_{n}}{\sigma_{n}}\right)}\phi(x;\theta_{n},\sigma_{n})1[a,b](x),\label{eq:truncated}\\
 & = & \frac{\phi(x;-n,(n+c)^{1/2})1[a,b](x)}{\Phi[-(a+n)(n+c)^{-1/2}]-\Phi[-(b+n)(n+c)^{-1/2}]}.\nonumber 
\end{eqnarray}
The normal density part equals
\[
\phi(x;-n,(n+c)^{1/2})=(2\pi)^{-1/2}(n+c)^{-1/2}\exp\left(-\frac{x^{2}+2nx+n^{2}}{2(n+c)}\right).
\]
When $n$ is large compared to $x$, the term $x^{2}/2(n+c)$ is negligible,
hence
\begin{eqnarray*}
\phi(x;-n,(n+c)^{1/2}) & \approx & (2\pi)^{-1/2}(n+c)^{-1/2}\exp(-n^{2}/2(n+c))\exp(-x),\\
 & = & (n+c)^{-1/2}\phi(n/(n+c)^{1/2})\exp(-x).
\end{eqnarray*}
From Equation 5 of \citet{borjesson1979simple}, we know that $\Phi(-x)\approx\phi(x)/x$
as $x$ grows. Then
\begin{equation*}
\Phi[-(a+n)(n+c)^{-1/2}] \approx \frac{(n+c)^{1/2}}{n+a}\phi[-(a+n)(n+c)^{-1/2}],
\end{equation*}
and using the same reasoning as above, we find that $\phi[-(a+n)(n+c)^{-1/2}]\approx\phi((n+c)^{1/2})\exp(-a)$
as $n$ increase. Therefore,
\[
\Phi[-(a+n)(n+c)^{-1/2}]\to\frac{(n+c)^{1/2}\phi((n+c)^{1/2})\exp(-a)}{a+n}.
\]
Since this reasoning applies to $b$ as well, we get that $f$ approaches
\begin{eqnarray*}
 &  & \frac{(n+c)^{-1/2}\phi(n/(n+c)^{1/2})\exp(-x)}{\Phi[-(a+n)n^{-1/2}]-\Phi[-(b+n)n^{-1/2}]},\\
 & \approx & \frac{(n+c)^{-1/2}\phi(n/(n+c)^{1/2})\exp(-x)}{(n+c)^{1/2}\phi((n+c)^{1/2})\left[\frac{\exp(-a)}{a+n}-\frac{\exp(-b)}{b+n}\right]},\\
 & = & \frac{\phi(-n^{2}/2(n+c))\exp(-x)}{\phi((n+c)^{1/2})\left[\frac{\exp(-a)}{a+n}-\frac{\exp(-b)}{b+n}\right]},\\
 & \approx & \exp(-x)n^{-1}\left[\frac{\exp(-a)}{a+n}-\frac{\exp(-b)}{b+n}\right]^{-1},\\
 & \to & \exp(-x)\left[\exp(-a)-\exp(-b)\right]^{-1}.
\end{eqnarray*}
Here the third line follows from $\phi(n/(n+c)^{1/2})/\phi((n+c)^{1/2})\to1$.
\end{proof}
These are the formula for the mixture probabilities
$\pi_{i}$, see \eqref{eq:mixture model formulation}. Let $c_{k}=\Phi^{-1}(1-\alpha_{k})$ and define 
\[
c=\sum_{k=1}^{K}\rho^{k}[\Phi(c_{k-1};\theta_{0},(\tau^{2}+\sigma_{i}^{2})^{1/2})-\Phi(c_{k};\theta_{0},(\tau^{2}+\sigma_{i}^{2})^{1/2})].
\]
Then \label{eq:pi_i formula}
\[
\pi^{k}=c^{-1}\rho^{k}[\Phi(c_{k-1};\theta_{0},(\tau^{2}+\sigma_{i}^{2})^{1/2})-\Phi(c_{k};\theta_{0},(\tau^{2}+\sigma_{i}^{2})^{1/2})].
\]

\bibliographystyle{apacite}
\bibliography{infiniteci}

\begin{thebibliography}{}

\bibitem [\protect \citeauthoryear {%
Bahadur%
\ \BBA {} Savage%
}{%
Bahadur%
\ \BBA {} Savage%
}{%
{\protect \APACyear {1956}}%
}]{%
bahadur1956nonexistence}
\APACinsertmetastar {%
bahadur1956nonexistence}%
\begin{APACrefauthors}%
Bahadur, R\BPBI R.%
\BCBT {}\ \BBA {} Savage, L\BPBI J.%
\end{APACrefauthors}%
\unskip\
\newblock
\APACrefYearMonthDay{1956}{}{}.
\newblock
{\BBOQ}\APACrefatitle {The Nonexistence of Certain Statistical Procedures in
  Nonparametric Problems} {The nonexistence of certain statistical procedures
  in nonparametric problems}.{\BBCQ}
\newblock
\APACjournalVolNumPages{Annals of Mathematical Statistics}{27}{4}{1115--1122}.
\newblock
\begin{APACrefDOI} \doi{10.1214/aoms/1177728077} \end{APACrefDOI}
\PrintBackRefs{\CurrentBib}

\bibitem [\protect \citeauthoryear {%
Berger%
, Liseo%
\BCBL {}\ \BBA {} Wolpert%
}{%
Berger%
\ \protect \BOthers {.}}{%
{\protect \APACyear {1999}}%
}]{%
berger1999integrated}
\APACinsertmetastar {%
berger1999integrated}%
\begin{APACrefauthors}%
Berger, J\BPBI O.%
, Liseo, B.%
\BCBL {}\ \BBA {} Wolpert, R\BPBI L.%
\end{APACrefauthors}%
\unskip\
\newblock
\APACrefYearMonthDay{1999}{}{}.
\newblock
{\BBOQ}\APACrefatitle {Integrated Likelihood Methods for Eliminating Nuisance
  Parameters} {Integrated likelihood methods for eliminating nuisance
  parameters}.{\BBCQ}
\newblock
\APACjournalVolNumPages{Statistical Science}{14}{1}{1--28}.
\newblock
\begin{APACrefDOI} \doi{10.1214/ss/1009211804} \end{APACrefDOI}
\PrintBackRefs{\CurrentBib}

\bibitem [\protect \citeauthoryear {%
Billingsley%
}{%
Billingsley%
}{%
{\protect \APACyear {1995}}%
}]{%
billingsley1995probability}
\APACinsertmetastar {%
billingsley1995probability}%
\begin{APACrefauthors}%
Billingsley, P.%
\end{APACrefauthors}%
\unskip\
\newblock
\APACrefYear{1995}.
\newblock
\APACrefbtitle {Probability and Measure} {Probability and measure}.
\newblock
\APACaddressPublisher{}{John Wiley \& Sons}.
\PrintBackRefs{\CurrentBib}

\bibitem [\protect \citeauthoryear {%
Borjesson%
\ \BBA {} Sundberg%
}{%
Borjesson%
\ \BBA {} Sundberg%
}{%
{\protect \APACyear {1979}}%
}]{%
borjesson1979simple}
\APACinsertmetastar {%
borjesson1979simple}%
\begin{APACrefauthors}%
Borjesson, P.%
\BCBT {}\ \BBA {} Sundberg, C.%
\end{APACrefauthors}%
\unskip\
\newblock
\APACrefYearMonthDay{1979}{}{}.
\newblock
{\BBOQ}\APACrefatitle {Simple Approximations of the Error Function {Q}(x) for
  Communications Applications} {Simple approximations of the error function
  {Q}(x) for communications applications}.{\BBCQ}
\newblock
\APACjournalVolNumPages{IEEE Transactions on Communications}{27}{3}{639--643}.
\newblock
\begin{APACrefDOI} \doi{10.1109/TCOM.1979.1094433} \end{APACrefDOI}
\PrintBackRefs{\CurrentBib}

\bibitem [\protect \citeauthoryear {%
Carter%
, Sch{\"o}nbrodt%
, Gervais%
\BCBL {}\ \BBA {} Hilgard%
}{%
Carter%
\ \protect \BOthers {.}}{%
{\protect \APACyear {2019}}%
}]{%
Carter2019-rw}
\APACinsertmetastar {%
Carter2019-rw}%
\begin{APACrefauthors}%
Carter, E\BPBI C.%
, Sch{\"o}nbrodt, F\BPBI D.%
, Gervais, W\BPBI M.%
\BCBL {}\ \BBA {} Hilgard, J.%
\end{APACrefauthors}%
\unskip\
\newblock
\APACrefYearMonthDay{2019}{}{}.
\newblock
{\BBOQ}\APACrefatitle {Correcting for Bias in Psychology: A Comparison of
  Meta-Analytic Methods} {Correcting for bias in psychology: A comparison of
  meta-analytic methods}.{\BBCQ}
\newblock
\APACjournalVolNumPages{Advances in Methods and Practices in Psychological
  Science}{2}{2}{115--144}.
\newblock
\begin{APACrefDOI} \doi{10.1177/2515245919847196} \end{APACrefDOI}
\PrintBackRefs{\CurrentBib}

\bibitem [\protect \citeauthoryear {%
Casella%
\ \BBA {} Berger%
}{%
Casella%
\ \BBA {} Berger%
}{%
{\protect \APACyear {2002}}%
}]{%
Casella2002-lg}
\APACinsertmetastar {%
Casella2002-lg}%
\begin{APACrefauthors}%
Casella, G.%
\BCBT {}\ \BBA {} Berger, R\BPBI L.%
\end{APACrefauthors}%
\unskip\
\newblock
\APACrefYear{2002}.
\newblock
\APACrefbtitle {Statistical Inference} {Statistical inference}.
\newblock
\APACaddressPublisher{}{Duxbury Press}.
\PrintBackRefs{\CurrentBib}

\bibitem [\protect \citeauthoryear {%
Citkowicz%
\ \BBA {} Vevea%
}{%
Citkowicz%
\ \BBA {} Vevea%
}{%
{\protect \APACyear {2017}}%
}]{%
citkowicz2017parsimonious}
\APACinsertmetastar {%
citkowicz2017parsimonious}%
\begin{APACrefauthors}%
Citkowicz, M.%
\BCBT {}\ \BBA {} Vevea, J\BPBI L.%
\end{APACrefauthors}%
\unskip\
\newblock
\APACrefYearMonthDay{2017}{}{}.
\newblock
{\BBOQ}\APACrefatitle {A Parsimonious Weight Function for Modeling Publication
  Bias} {A parsimonious weight function for modeling publication bias}.{\BBCQ}
\newblock
\APACjournalVolNumPages{Psychological Methods}{22}{1}{28--41}.
\newblock
\begin{APACrefDOI} \doi{10.1037/met0000119} \end{APACrefDOI}
\PrintBackRefs{\CurrentBib}

\bibitem [\protect \citeauthoryear {%
Coburn%
, Vevea%
\BCBL {}\ \BBA {} {Coburn}%
}{%
Coburn%
\ \protect \BOthers {.}}{%
{\protect \APACyear {2019}}%
}]{%
Coburn2019-ec}
\APACinsertmetastar {%
Coburn2019-ec}%
\begin{APACrefauthors}%
Coburn, K\BPBI M.%
, Vevea, J\BPBI L.%
\BCBL {}\ \BBA {} {Coburn}.%
\end{APACrefauthors}%
\unskip\
\newblock
\APACrefYearMonthDay{2019}{}{}.
\newblock
{\BBOQ}\APACrefatitle {{R} package `weightr'} {{R} package `weightr'}{\BBCQ}\
  [\bibcomputersoftwaremanual].
\newblock
\APACaddressPublisher{}{cran.rstudio.org}.
\newblock
\begin{APACrefURL}
  \url{https://cran.rstudio.org/web/packages/weightr/weightr.pdf}
  \end{APACrefURL}
\PrintBackRefs{\CurrentBib}

\bibitem [\protect \citeauthoryear {%
Copas%
}{%
Copas%
}{%
{\protect \APACyear {2013}}%
}]{%
Copas2013-qi}
\APACinsertmetastar {%
Copas2013-qi}%
\begin{APACrefauthors}%
Copas, J\BPBI B.%
\end{APACrefauthors}%
\unskip\
\newblock
\APACrefYearMonthDay{2013}{}{}.
\newblock
{\BBOQ}\APACrefatitle {A likelihood-based sensitivity analysis for publication
  bias in meta-analysis} {A likelihood-based sensitivity analysis for
  publication bias in meta-analysis}.{\BBCQ}
\newblock
\APACjournalVolNumPages{Journal of the Royal Statistical Society. Series C,
  Applied statistics}{62}{1}{47--66}.
\newblock
\begin{APACrefDOI} \doi{10.1111/j.1467-9876.2012.01049.x} \end{APACrefDOI}
\PrintBackRefs{\CurrentBib}

\bibitem [\protect \citeauthoryear {%
Cuddy%
, Schultz%
\BCBL {}\ \BBA {} Fosse%
}{%
Cuddy%
\ \protect \BOthers {.}}{%
{\protect \APACyear {2018}}%
}]{%
Cuddy2018-kp}
\APACinsertmetastar {%
Cuddy2018-kp}%
\begin{APACrefauthors}%
Cuddy, A\BPBI J\BPBI C.%
, Schultz, S\BPBI J.%
\BCBL {}\ \BBA {} Fosse, N\BPBI E.%
\end{APACrefauthors}%
\unskip\
\newblock
\APACrefYearMonthDay{2018}{}{}.
\newblock
{\BBOQ}\APACrefatitle {{P-Curving} a More Comprehensive Body of Research on
  Postural Feedback Reveals Clear Evidential Value for Power-Posing Effects:
  Reply to {Simmons} and {Simonsohn} (2017)} {{P-Curving} a more comprehensive
  body of research on postural feedback reveals clear evidential value for
  power-posing effects: Reply to {Simmons} and {Simonsohn} (2017)}.{\BBCQ}
\newblock
\APACjournalVolNumPages{Psychological Science}{29}{4}{656--666}.
\newblock
\begin{APACrefDOI} \doi{10.1177/0956797617746749} \end{APACrefDOI}
\PrintBackRefs{\CurrentBib}

\bibitem [\protect \citeauthoryear {%
Dear%
\ \BBA {} Begg%
}{%
Dear%
\ \BBA {} Begg%
}{%
{\protect \APACyear {1992}}%
}]{%
Dear1992-gw}
\APACinsertmetastar {%
Dear1992-gw}%
\begin{APACrefauthors}%
Dear, K\BPBI B\BPBI G.%
\BCBT {}\ \BBA {} Begg, C\BPBI B.%
\end{APACrefauthors}%
\unskip\
\newblock
\APACrefYearMonthDay{1992}{}{}.
\newblock
{\BBOQ}\APACrefatitle {An Approach for Assessing Publication Bias Prior to
  Performing a {Meta-analysis}} {An approach for assessing publication bias
  prior to performing a {Meta-analysis}}.{\BBCQ}
\newblock
\APACjournalVolNumPages{Statistical Science}{7}{2}{237--245}.
\newblock
\begin{APACrefDOI} \doi{10.1214/ss/1177011363} \end{APACrefDOI}
\PrintBackRefs{\CurrentBib}

\bibitem [\protect \citeauthoryear {%
Donoho%
}{%
Donoho%
}{%
{\protect \APACyear {1988}}%
}]{%
Donoho1988-hg}
\APACinsertmetastar {%
Donoho1988-hg}%
\begin{APACrefauthors}%
Donoho, D\BPBI L.%
\end{APACrefauthors}%
\unskip\
\newblock
\APACrefYearMonthDay{1988}{}{}.
\newblock
{\BBOQ}\APACrefatitle {One-Sided Inference about Functionals of a Density}
  {One-sided inference about functionals of a density}.{\BBCQ}
\newblock
\APACjournalVolNumPages{Annals of Statistics}{16}{4}{1390--1420}.
\newblock
\begin{APACrefDOI} \doi{10.1214/aos/1176351045} \end{APACrefDOI}
\PrintBackRefs{\CurrentBib}

\bibitem [\protect \citeauthoryear {%
Duval%
\ \BBA {} Tweedie%
}{%
Duval%
\ \BBA {} Tweedie%
}{%
{\protect \APACyear {2000}}%
}]{%
Duval2000-ct}
\APACinsertmetastar {%
Duval2000-ct}%
\begin{APACrefauthors}%
Duval, S.%
\BCBT {}\ \BBA {} Tweedie, R.%
\end{APACrefauthors}%
\unskip\
\newblock
\APACrefYearMonthDay{2000}{}{}.
\newblock
{\BBOQ}\APACrefatitle {A Nonparametric ``Trim and Fill'' Method of Accounting
  for Publication Bias in Meta-analysis} {A nonparametric ``trim and fill''
  method of accounting for publication bias in meta-analysis}.{\BBCQ}
\newblock
\APACjournalVolNumPages{Journal of the American Statistical
  Association}{95}{449}{89--98}.
\newblock
\begin{APACrefDOI} \doi{10.1080/01621459.2000.10473905} \end{APACrefDOI}
\PrintBackRefs{\CurrentBib}

\bibitem [\protect \citeauthoryear {%
Easterbrook%
, Berlin%
, Gopalan%
\BCBL {}\ \BBA {} Matthews%
}{%
Easterbrook%
\ \protect \BOthers {.}}{%
{\protect \APACyear {1991}}%
}]{%
Easterbrook1991-ph}
\APACinsertmetastar {%
Easterbrook1991-ph}%
\begin{APACrefauthors}%
Easterbrook, P\BPBI J.%
, Berlin, J\BPBI A.%
, Gopalan, R.%
\BCBL {}\ \BBA {} Matthews, D\BPBI R.%
\end{APACrefauthors}%
\unskip\
\newblock
\APACrefYearMonthDay{1991}{}{}.
\newblock
{\BBOQ}\APACrefatitle {Publication Bias in Clinical Research} {Publication bias
  in clinical research}.{\BBCQ}
\newblock
\APACjournalVolNumPages{The Lancet}{337}{8746}{867--872}.
\newblock
\begin{APACrefDOI} \doi{10.1016/0140-6736(91)90201-y} \end{APACrefDOI}
\PrintBackRefs{\CurrentBib}

\bibitem [\protect \citeauthoryear {%
Egger%
, Davey~Smith%
, Schneider%
\BCBL {}\ \BBA {} Minder%
}{%
Egger%
\ \protect \BOthers {.}}{%
{\protect \APACyear {1997}}%
}]{%
Egger1997-rd}
\APACinsertmetastar {%
Egger1997-rd}%
\begin{APACrefauthors}%
Egger, M.%
, Davey~Smith, G.%
, Schneider, M.%
\BCBL {}\ \BBA {} Minder, C.%
\end{APACrefauthors}%
\unskip\
\newblock
\APACrefYearMonthDay{1997}{}{}.
\newblock
{\BBOQ}\APACrefatitle {Bias in Meta-analysis Detected by a Simple, Graphical
  Test} {Bias in meta-analysis detected by a simple, graphical test}.{\BBCQ}
\newblock
\APACjournalVolNumPages{British Medical Journal}{315}{7109}{629--634}.
\newblock
\begin{APACrefDOI} \doi{10.1136/bmj.315.7109.629} \end{APACrefDOI}
\PrintBackRefs{\CurrentBib}

\bibitem [\protect \citeauthoryear {%
Egger%
\ \BBA {} Smith%
}{%
Egger%
\ \BBA {} Smith%
}{%
{\protect \APACyear {1998}}%
}]{%
Egger1998-kj}
\APACinsertmetastar {%
Egger1998-kj}%
\begin{APACrefauthors}%
Egger, M.%
\BCBT {}\ \BBA {} Smith, G\BPBI D.%
\end{APACrefauthors}%
\unskip\
\newblock
\APACrefYearMonthDay{1998}{}{}.
\newblock
{\BBOQ}\APACrefatitle {Bias in Location and Selection of Studies} {Bias in
  location and selection of studies}.{\BBCQ}
\newblock
\APACjournalVolNumPages{British Medical Journal}{316}{7124}{61--66}.
\newblock
\begin{APACrefDOI} \doi{10.1136/bmj.316.7124.61} \end{APACrefDOI}
\PrintBackRefs{\CurrentBib}

\bibitem [\protect \citeauthoryear {%
Fieller%
}{%
Fieller%
}{%
{\protect \APACyear {1940}}%
}]{%
Fieller1940-lg}
\APACinsertmetastar {%
Fieller1940-lg}%
\begin{APACrefauthors}%
Fieller, E\BPBI C.%
\end{APACrefauthors}%
\unskip\
\newblock
\APACrefYearMonthDay{1940}{}{}.
\newblock
{\BBOQ}\APACrefatitle {The Biological Standardization of Insulin} {The
  biological standardization of insulin}.{\BBCQ}
\newblock
\APACjournalVolNumPages{Supplement to the Journal of the Royal Statistical
  Society}{7}{1}{1--64}.
\newblock
\begin{APACrefDOI} \doi{10.2307/2983630} \end{APACrefDOI}
\PrintBackRefs{\CurrentBib}

\bibitem [\protect \citeauthoryear {%
Flury%
}{%
Flury%
}{%
{\protect \APACyear {1990}}%
}]{%
flury1990acceptance}
\APACinsertmetastar {%
flury1990acceptance}%
\begin{APACrefauthors}%
Flury, B\BPBI D.%
\end{APACrefauthors}%
\unskip\
\newblock
\APACrefYearMonthDay{1990}{}{}.
\newblock
{\BBOQ}\APACrefatitle {{Acceptance}--Rejection Sampling Made Easy}
  {{Acceptance}--rejection sampling made easy}.{\BBCQ}
\newblock
\APACjournalVolNumPages{SIAM Review}{32}{3}{474--476}.
\newblock
\begin{APACrefDOI} \doi{10.1137/1032082} \end{APACrefDOI}
\PrintBackRefs{\CurrentBib}

\bibitem [\protect \citeauthoryear {%
Gleser%
\ \BBA {} Hwang%
}{%
Gleser%
\ \BBA {} Hwang%
}{%
{\protect \APACyear {1987}}%
}]{%
gleser1987nonexistence}
\APACinsertmetastar {%
gleser1987nonexistence}%
\begin{APACrefauthors}%
Gleser, L\BPBI J.%
\BCBT {}\ \BBA {} Hwang, J\BPBI T.%
\end{APACrefauthors}%
\unskip\
\newblock
\APACrefYearMonthDay{1987}{}{}.
\newblock
{\BBOQ}\APACrefatitle {The Nonexistence of 100(1 - $\alpha$)\% Confidence Sets
  of Finite Expected Diameter in Errors-in-Variables and Related Models} {The
  nonexistence of 100(1 - $\alpha$)\% confidence sets of finite expected
  diameter in errors-in-variables and related models}.{\BBCQ}
\newblock
\APACjournalVolNumPages{Annals of Statistics}{15}{4}{1351--1362}.
\newblock
\begin{APACrefDOI} \doi{10.1214/aos/1176350597} \end{APACrefDOI}
\PrintBackRefs{\CurrentBib}

\bibitem [\protect \citeauthoryear {%
Hedges%
}{%
Hedges%
}{%
{\protect \APACyear {1984}}%
}]{%
hedges1984estimation}
\APACinsertmetastar {%
hedges1984estimation}%
\begin{APACrefauthors}%
Hedges, L\BPBI V.%
\end{APACrefauthors}%
\unskip\
\newblock
\APACrefYearMonthDay{1984}{}{}.
\newblock
{\BBOQ}\APACrefatitle {Estimation of Effect Size under Nonrandom Sampling: The
  Effects of Censoring Studies Yielding Statistically Insignificant Mean
  Differences} {Estimation of effect size under nonrandom sampling: The effects
  of censoring studies yielding statistically insignificant mean
  differences}.{\BBCQ}
\newblock
\APACjournalVolNumPages{Journal of Educational and Behavioral
  Statistics}{9}{1}{61--85}.
\newblock
\begin{APACrefDOI} \doi{10.3102/10769986009001061} \end{APACrefDOI}
\PrintBackRefs{\CurrentBib}

\bibitem [\protect \citeauthoryear {%
Hedges%
}{%
Hedges%
}{%
{\protect \APACyear {1992}}%
}]{%
hedges1992modeling}
\APACinsertmetastar {%
hedges1992modeling}%
\begin{APACrefauthors}%
Hedges, L\BPBI V.%
\end{APACrefauthors}%
\unskip\
\newblock
\APACrefYearMonthDay{1992}{}{}.
\newblock
{\BBOQ}\APACrefatitle {Modeling Publication Selection Effects in Meta-analysis}
  {Modeling publication selection effects in meta-analysis}.{\BBCQ}
\newblock
\APACjournalVolNumPages{Statistical Science}{7}{2}{246--255}.
\newblock
\begin{APACrefDOI} \doi{10.1214/ss/1177011364} \end{APACrefDOI}
\PrintBackRefs{\CurrentBib}

\bibitem [\protect \citeauthoryear {%
Hedges%
\ \BBA {} Vevea%
}{%
Hedges%
\ \BBA {} Vevea%
}{%
{\protect \APACyear {1998}}%
}]{%
hedges1998fixed}
\APACinsertmetastar {%
hedges1998fixed}%
\begin{APACrefauthors}%
Hedges, L\BPBI V.%
\BCBT {}\ \BBA {} Vevea, J\BPBI L.%
\end{APACrefauthors}%
\unskip\
\newblock
\APACrefYearMonthDay{1998}{}{}.
\newblock
{\BBOQ}\APACrefatitle {Fixed- and Random-Effects Models in Meta-analysis}
  {Fixed- and random-effects models in meta-analysis}.{\BBCQ}
\newblock
\APACjournalVolNumPages{Psychological Methods}{3}{4}{486}.
\newblock
\begin{APACrefDOI} \doi{10.1037/1082-989x.3.4.486} \end{APACrefDOI}
\PrintBackRefs{\CurrentBib}

\bibitem [\protect \citeauthoryear {%
Higgins%
\ \protect \BOthers {.}}{%
Higgins%
\ \protect \BOthers {.}}{%
{\protect \APACyear {2019}}%
}]{%
Higgins2019-vv}
\APACinsertmetastar {%
Higgins2019-vv}%
\begin{APACrefauthors}%
Higgins, J\BPBI P\BPBI T.%
, Thomas, J.%
, Chandler, J.%
, Cumpston, M.%
, Li, T.%
, Page, M\BPBI J.%
\BCBL {}\ \BBA {} Welch, V\BPBI A.%
\end{APACrefauthors}%
\unskip\
\newblock
\APACrefYear{2019}.
\newblock
\APACrefbtitle {Cochrane Handbook for Systematic Reviews of Interventions}
  {Cochrane handbook for systematic reviews of interventions}.
\newblock
\APACaddressPublisher{}{John Wiley \& Sons}.
\PrintBackRefs{\CurrentBib}

\bibitem [\protect \citeauthoryear {%
Iyengar%
\ \BBA {} Greenhouse%
}{%
Iyengar%
\ \BBA {} Greenhouse%
}{%
{\protect \APACyear {1988}}%
}]{%
iyengar1988selection}
\APACinsertmetastar {%
iyengar1988selection}%
\begin{APACrefauthors}%
Iyengar, S.%
\BCBT {}\ \BBA {} Greenhouse, J\BPBI B.%
\end{APACrefauthors}%
\unskip\
\newblock
\APACrefYearMonthDay{1988}{}{}.
\newblock
{\BBOQ}\APACrefatitle {Selection Models and the File Drawer Problem} {Selection
  models and the file drawer problem}.{\BBCQ}
\newblock
\APACjournalVolNumPages{Statistical Science}{3}{1}{109--117}.
\newblock
\begin{APACrefDOI} \doi{10.1214/ss/1177013012} \end{APACrefDOI}
\PrintBackRefs{\CurrentBib}

\bibitem [\protect \citeauthoryear {%
Koschat%
}{%
Koschat%
}{%
{\protect \APACyear {1987}}%
}]{%
Koschat1987-dk}
\APACinsertmetastar {%
Koschat1987-dk}%
\begin{APACrefauthors}%
Koschat, M\BPBI A.%
\end{APACrefauthors}%
\unskip\
\newblock
\APACrefYearMonthDay{1987}{}{}.
\newblock
{\BBOQ}\APACrefatitle {A Characterization of the Fieller Solution} {A
  characterization of the fieller solution}.{\BBCQ}
\newblock
\APACjournalVolNumPages{Annals of Statistics}{15}{1}{462--468}.
\newblock
\begin{APACrefDOI} \doi{10.1214/aos/1176350282} \end{APACrefDOI}
\PrintBackRefs{\CurrentBib}

\bibitem [\protect \citeauthoryear {%
Lehmann%
\ \BBA {} Romano%
}{%
Lehmann%
\ \BBA {} Romano%
}{%
{\protect \APACyear {2005}}%
}]{%
lehmann2006testing}
\APACinsertmetastar {%
lehmann2006testing}%
\begin{APACrefauthors}%
Lehmann, E\BPBI L.%
\BCBT {}\ \BBA {} Romano, J\BPBI P.%
\end{APACrefauthors}%
\unskip\
\newblock
\APACrefYear{2005}.
\newblock
\APACrefbtitle {Testing Statistical Hypotheses} {Testing statistical
  hypotheses}.
\newblock
\APACaddressPublisher{}{Springer Science \& Business Media}.
\newblock
\begin{APACrefDOI} \doi{10.1007/0-387-27605-x} \end{APACrefDOI}
\PrintBackRefs{\CurrentBib}

\bibitem [\protect \citeauthoryear {%
McShane%
, B{\"o}ckenholt%
\BCBL {}\ \BBA {} Hansen%
}{%
McShane%
\ \protect \BOthers {.}}{%
{\protect \APACyear {2016}}%
}]{%
McShane2016-rb}
\APACinsertmetastar {%
McShane2016-rb}%
\begin{APACrefauthors}%
McShane, B\BPBI B.%
, B{\"o}ckenholt, U.%
\BCBL {}\ \BBA {} Hansen, K\BPBI T.%
\end{APACrefauthors}%
\unskip\
\newblock
\APACrefYearMonthDay{2016}{}{}.
\newblock
{\BBOQ}\APACrefatitle {Adjusting for Publication Bias in {Meta-analysis}: An
  Evaluation of Selection Methods and Some Cautionary Notes} {Adjusting for
  publication bias in {Meta-analysis}: An evaluation of selection methods and
  some cautionary notes}.{\BBCQ}
\newblock
\APACjournalVolNumPages{Perspectives on Psychological
  Science}{11}{5}{730--749}.
\newblock
\begin{APACrefDOI} \doi{10.1177/1745691616662243} \end{APACrefDOI}
\PrintBackRefs{\CurrentBib}

\bibitem [\protect \citeauthoryear {%
Moss%
\ \BBA {} De~Bin%
}{%
Moss%
\ \BBA {} De~Bin%
}{%
{\protect \APACyear {2021}}%
}]{%
moss2019modelling}
\APACinsertmetastar {%
moss2019modelling}%
\begin{APACrefauthors}%
Moss, J.%
\BCBT {}\ \BBA {} De~Bin, R.%
\end{APACrefauthors}%
\unskip\
\newblock
\APACrefYearMonthDay{2021}{{\APACmonth{09}}}{}.
\newblock
{\BBOQ}\APACrefatitle {Modelling publication bias and p-hacking} {Modelling
  publication bias and p-hacking}.{\BBCQ}
\newblock
\APACjournalVolNumPages{Biometrics}{}{}{}.
\newblock
\begin{APACrefURL} \url{http://dx.doi.org/10.1111/biom.13560} \end{APACrefURL}
\newblock
\begin{APACrefDOI} \doi{10.1111/biom.13560} \end{APACrefDOI}
\PrintBackRefs{\CurrentBib}

\bibitem [\protect \citeauthoryear {%
Motyl%
\ \protect \BOthers {.}}{%
Motyl%
\ \protect \BOthers {.}}{%
{\protect \APACyear {2017}}%
}]{%
Motyl2017-dx}
\APACinsertmetastar {%
Motyl2017-dx}%
\begin{APACrefauthors}%
Motyl, M.%
, Demos, A\BPBI P.%
, Carsel, T\BPBI S.%
, Hanson, B\BPBI E.%
, Melton, Z\BPBI J.%
, Mueller, A\BPBI B.%
\BDBL {}Skitka, L\BPBI J.%
\end{APACrefauthors}%
\unskip\
\newblock
\APACrefYearMonthDay{2017}{}{}.
\newblock
{\BBOQ}\APACrefatitle {The State of Social and Personality Science: Rotten to
  the Core, Not So Bad, Getting Better, or Getting Worse?} {The state of social
  and personality science: Rotten to the core, not so bad, getting better, or
  getting worse?}{\BBCQ}
\newblock
\APACjournalVolNumPages{Journal of Personality and Social
  Psychology}{113}{1}{34--58}.
\newblock
\begin{APACrefDOI} \doi{10.1037/pspa0000084} \end{APACrefDOI}
\PrintBackRefs{\CurrentBib}

\bibitem [\protect \citeauthoryear {%
Preston%
, Ashby%
\BCBL {}\ \BBA {} Smyth%
}{%
Preston%
\ \protect \BOthers {.}}{%
{\protect \APACyear {2004}}%
}]{%
Preston2004-pv}
\APACinsertmetastar {%
Preston2004-pv}%
\begin{APACrefauthors}%
Preston, C.%
, Ashby, D.%
\BCBL {}\ \BBA {} Smyth, R.%
\end{APACrefauthors}%
\unskip\
\newblock
\APACrefYearMonthDay{2004}{}{}.
\newblock
{\BBOQ}\APACrefatitle {Adjusting for publication bias: modelling the selection
  process} {Adjusting for publication bias: modelling the selection
  process}.{\BBCQ}
\newblock
\APACjournalVolNumPages{Journal of Evaluation in Clinical
  Practice}{10}{2}{313--322}.
\newblock
\begin{APACrefDOI} \doi{10.1111/j.1365-2753.2003.00457.x} \end{APACrefDOI}
\PrintBackRefs{\CurrentBib}

\bibitem [\protect \citeauthoryear {%
{R Core Team}%
}{%
{R Core Team}%
}{%
{\protect \APACyear {2020}}%
}]{%
Team2013-tt}
\APACinsertmetastar {%
Team2013-tt}%
\begin{APACrefauthors}%
{R Core Team}.%
\end{APACrefauthors}%
\unskip\
\newblock
\APACrefYearMonthDay{2020}{}{}.
\newblock
{\BBOQ}\APACrefatitle {R: A Language and Environment for Statistical Computing}
  {R: A language and environment for statistical computing}{\BBCQ}\
  [\bibcomputersoftwaremanual].
\newblock
\APACaddressPublisher{Vienna, Austria}{}.
\PrintBackRefs{\CurrentBib}

\bibitem [\protect \citeauthoryear {%
Romano%
}{%
Romano%
}{%
{\protect \APACyear {2004}}%
}]{%
romano2004non}
\APACinsertmetastar {%
romano2004non}%
\begin{APACrefauthors}%
Romano, J\BPBI P.%
\end{APACrefauthors}%
\unskip\
\newblock
\APACrefYearMonthDay{2004}{}{}.
\newblock
{\BBOQ}\APACrefatitle {On Non-parametric Testing, the Uniform Behaviour of the
  t-test, and Related Problems} {On non-parametric testing, the uniform
  behaviour of the t-test, and related problems}.{\BBCQ}
\newblock
\APACjournalVolNumPages{Scandinavian Journal of Statistics}{31}{4}{567--584}.
\newblock
\begin{APACrefDOI} \doi{10.1111/j.1467-9469.2004.00407.x} \end{APACrefDOI}
\PrintBackRefs{\CurrentBib}

\bibitem [\protect \citeauthoryear {%
Rosenthal%
}{%
Rosenthal%
}{%
{\protect \APACyear {1979}}%
}]{%
Rosenthal1979-pm}
\APACinsertmetastar {%
Rosenthal1979-pm}%
\begin{APACrefauthors}%
Rosenthal, R.%
\end{APACrefauthors}%
\unskip\
\newblock
\APACrefYearMonthDay{1979}{}{}.
\newblock
{\BBOQ}\APACrefatitle {The File Drawer Problem and Tolerance for Null Results}
  {The file drawer problem and tolerance for null results}.{\BBCQ}
\newblock
\APACjournalVolNumPages{Psychological Bulletin}{86}{3}{638--641}.
\newblock
\begin{APACrefDOI} \doi{10.1037/0033-2909.86.3.638} \end{APACrefDOI}
\PrintBackRefs{\CurrentBib}

\bibitem [\protect \citeauthoryear {%
Simmons%
, Nelson%
\BCBL {}\ \BBA {} Simonsohn%
}{%
Simmons%
\ \protect \BOthers {.}}{%
{\protect \APACyear {2011}}%
}]{%
simmons2011false}
\APACinsertmetastar {%
simmons2011false}%
\begin{APACrefauthors}%
Simmons, J\BPBI P.%
, Nelson, L\BPBI D.%
\BCBL {}\ \BBA {} Simonsohn, U.%
\end{APACrefauthors}%
\unskip\
\newblock
\APACrefYearMonthDay{2011}{}{}.
\newblock
{\BBOQ}\APACrefatitle {False-positive Psychology: {Undisclosed} Flexibility in
  Data Collection and Analysis Allows Presenting Anything as Significant}
  {False-positive psychology: {Undisclosed} flexibility in data collection and
  analysis allows presenting anything as significant}.{\BBCQ}
\newblock
\APACjournalVolNumPages{Psychological Science}{22}{11}{1359--1366}.
\newblock
\begin{APACrefDOI} \doi{10.1177/0956797611417632} \end{APACrefDOI}
\PrintBackRefs{\CurrentBib}

\bibitem [\protect \citeauthoryear {%
Simonsohn%
, Nelson%
\BCBL {}\ \BBA {} Simmons%
}{%
Simonsohn%
\ \protect \BOthers {.}}{%
{\protect \APACyear {2014}}%
}]{%
Simonsohn2014-cn}
\APACinsertmetastar {%
Simonsohn2014-cn}%
\begin{APACrefauthors}%
Simonsohn, U.%
, Nelson, L\BPBI D.%
\BCBL {}\ \BBA {} Simmons, J\BPBI P.%
\end{APACrefauthors}%
\unskip\
\newblock
\APACrefYearMonthDay{2014}{}{}.
\newblock
{\BBOQ}\APACrefatitle {p-Curve and Effect Size: Correcting for Publication Bias
  Using Only Significant Results} {p-curve and effect size: Correcting for
  publication bias using only significant results}.{\BBCQ}
\newblock
\APACjournalVolNumPages{Perspectives on Psychological Science}{9}{6}{666--681}.
\newblock
\begin{APACrefDOI} \doi{10.1177/1745691614553988} \end{APACrefDOI}
\PrintBackRefs{\CurrentBib}

\bibitem [\protect \citeauthoryear {%
Stanley%
}{%
Stanley%
}{%
{\protect \APACyear {2005}}%
}]{%
Stanley2005-ng}
\APACinsertmetastar {%
Stanley2005-ng}%
\begin{APACrefauthors}%
Stanley, T\BPBI D.%
\end{APACrefauthors}%
\unskip\
\newblock
\APACrefYearMonthDay{2005}{}{}.
\newblock
{\BBOQ}\APACrefatitle {Beyond Publication Bias} {Beyond publication
  bias}.{\BBCQ}
\newblock
\APACjournalVolNumPages{Journal of Economic Surveys}{19}{3}{309--345}.
\newblock
\begin{APACrefDOI} \doi{10.1111/j.0950-0804.2005.00250.x} \end{APACrefDOI}
\PrintBackRefs{\CurrentBib}

\bibitem [\protect \citeauthoryear {%
Stanley%
\ \BBA {} Doucouliagos%
}{%
Stanley%
\ \BBA {} Doucouliagos%
}{%
{\protect \APACyear {2014}}%
}]{%
stanley2014meta}
\APACinsertmetastar {%
stanley2014meta}%
\begin{APACrefauthors}%
Stanley, T\BPBI D.%
\BCBT {}\ \BBA {} Doucouliagos, H.%
\end{APACrefauthors}%
\unskip\
\newblock
\APACrefYearMonthDay{2014}{}{}.
\newblock
{\BBOQ}\APACrefatitle {Meta-Regression Approximations to Reduce Publication
  Selection Bias} {Meta-regression approximations to reduce publication
  selection bias}.{\BBCQ}
\newblock
\APACjournalVolNumPages{Research Synthesis Methods}{5}{1}{60--78}.
\newblock
\begin{APACrefDOI} \doi{10.1002/jrsm.1095} \end{APACrefDOI}
\PrintBackRefs{\CurrentBib}

\bibitem [\protect \citeauthoryear {%
Sterling%
}{%
Sterling%
}{%
{\protect \APACyear {1959}}%
}]{%
sterling1959publication}
\APACinsertmetastar {%
sterling1959publication}%
\begin{APACrefauthors}%
Sterling, T\BPBI D.%
\end{APACrefauthors}%
\unskip\
\newblock
\APACrefYearMonthDay{1959}{}{}.
\newblock
{\BBOQ}\APACrefatitle {Publication Decisions and their Possible Effects on
  Inferences Drawn from Tests of Significance---or Vice Versa} {Publication
  decisions and their possible effects on inferences drawn from tests of
  significance---or vice versa}.{\BBCQ}
\newblock
\APACjournalVolNumPages{Journal of the American Statistical
  Association}{54}{285}{30--34}.
\newblock
\begin{APACrefDOI} \doi{10.1080/01621459.1959.10501497} \end{APACrefDOI}
\PrintBackRefs{\CurrentBib}

\bibitem [\protect \citeauthoryear {%
van Aert%
, Wicherts%
\BCBL {}\ \BBA {} van Assen%
}{%
van Aert%
\ \protect \BOthers {.}}{%
{\protect \APACyear {2016}}%
}]{%
Van_Aert2016-cu}
\APACinsertmetastar {%
Van_Aert2016-cu}%
\begin{APACrefauthors}%
van Aert, R\BPBI C\BPBI M.%
, Wicherts, J\BPBI M.%
\BCBL {}\ \BBA {} van Assen, M\BPBI A\BPBI L\BPBI M.%
\end{APACrefauthors}%
\unskip\
\newblock
\APACrefYearMonthDay{2016}{}{}.
\newblock
{\BBOQ}\APACrefatitle {Conducting Meta-Analyses Based on p Values: Reservations
  and Recommendations for Applying p-Uniform and p-Curve} {Conducting
  meta-analyses based on p values: Reservations and recommendations for
  applying p-uniform and p-curve}.{\BBCQ}
\newblock
\APACjournalVolNumPages{Perspectives on Psychological
  Science}{11}{5}{713--729}.
\newblock
\begin{APACrefDOI} \doi{10.1177/1745691616650874} \end{APACrefDOI}
\PrintBackRefs{\CurrentBib}

\bibitem [\protect \citeauthoryear {%
van Assen%
, van Aert%
\BCBL {}\ \BBA {} Wicherts%
}{%
van Assen%
\ \protect \BOthers {.}}{%
{\protect \APACyear {2015}}%
}]{%
Van_Assen2015-qs}
\APACinsertmetastar {%
Van_Assen2015-qs}%
\begin{APACrefauthors}%
van Assen, M\BPBI A\BPBI L\BPBI M.%
, van Aert, R\BPBI C\BPBI M.%
\BCBL {}\ \BBA {} Wicherts, J\BPBI M.%
\end{APACrefauthors}%
\unskip\
\newblock
\APACrefYearMonthDay{2015}{}{}.
\newblock
{\BBOQ}\APACrefatitle {Meta-analysis Using Effect Size Distributions of Only
  Statistically Significant Studies} {Meta-analysis using effect size
  distributions of only statistically significant studies}.{\BBCQ}
\newblock
\APACjournalVolNumPages{Psychological Methods}{20}{3}{293--309}.
\newblock
\begin{APACrefDOI} \doi{10.1037/met0000025} \end{APACrefDOI}
\PrintBackRefs{\CurrentBib}

\bibitem [\protect \citeauthoryear {%
Vevea%
\ \BBA {} Hedges%
}{%
Vevea%
\ \BBA {} Hedges%
}{%
{\protect \APACyear {1995}}%
}]{%
Vevea1995-on}
\APACinsertmetastar {%
Vevea1995-on}%
\begin{APACrefauthors}%
Vevea, J\BPBI L.%
\BCBT {}\ \BBA {} Hedges, L\BPBI V.%
\end{APACrefauthors}%
\unskip\
\newblock
\APACrefYearMonthDay{1995}{}{}.
\newblock
{\BBOQ}\APACrefatitle {A General Linear Model for Estimating Effect Size in the
  Presence of Bublication Bias} {A general linear model for estimating effect
  size in the presence of bublication bias}.{\BBCQ}
\newblock
\APACjournalVolNumPages{Psychometrika}{60}{3}{419--435}.
\newblock
\begin{APACrefDOI} \doi{10.1007/bf02294384} \end{APACrefDOI}
\PrintBackRefs{\CurrentBib}

\bibitem [\protect \citeauthoryear {%
Vevea%
\ \BBA {} Woods%
}{%
Vevea%
\ \BBA {} Woods%
}{%
{\protect \APACyear {2005}}%
}]{%
Vevea2005-xp}
\APACinsertmetastar {%
Vevea2005-xp}%
\begin{APACrefauthors}%
Vevea, J\BPBI L.%
\BCBT {}\ \BBA {} Woods, C\BPBI M.%
\end{APACrefauthors}%
\unskip\
\newblock
\APACrefYearMonthDay{2005}{}{}.
\newblock
{\BBOQ}\APACrefatitle {Publication Bias in Research Synthesis: Sensitivity
  Analysis Using A Priori Weight Functions} {Publication bias in research
  synthesis: Sensitivity analysis using a priori weight functions}.{\BBCQ}
\newblock
\APACjournalVolNumPages{Psychological Methods}{10}{4}{428--443}.
\newblock
\begin{APACrefDOI} \doi{10.1037/1082-989X.10.4.428} \end{APACrefDOI}
\PrintBackRefs{\CurrentBib}

\bibitem [\protect \citeauthoryear {%
Viechtbauer%
}{%
Viechtbauer%
}{%
{\protect \APACyear {2010}}%
}]{%
Viechtbauer2010-cf}
\APACinsertmetastar {%
Viechtbauer2010-cf}%
\begin{APACrefauthors}%
Viechtbauer, W.%
\end{APACrefauthors}%
\unskip\
\newblock
\APACrefYearMonthDay{2010}{}{}.
\newblock
{\BBOQ}\APACrefatitle {Conducting meta-analyses in {R} with the metafor
  package} {Conducting meta-analyses in {R} with the metafor package}.{\BBCQ}
\newblock
\APACjournalVolNumPages{Journal of Statistical Software}{36}{3}{1--48}.
\newblock
\begin{APACrefDOI} \doi{10.18637/jss.v036.i03} \end{APACrefDOI}
\PrintBackRefs{\CurrentBib}

\bibitem [\protect \citeauthoryear {%
von Neumann%
}{%
von Neumann%
}{%
{\protect \APACyear {1951}}%
}]{%
von1951various}
\APACinsertmetastar {%
von1951various}%
\begin{APACrefauthors}%
von Neumann, J.%
\end{APACrefauthors}%
\unskip\
\newblock
\APACrefYearMonthDay{1951}{}{}.
\newblock
{\BBOQ}\APACrefatitle {Various Techniques Used in Connection with Random
  Digits} {Various techniques used in connection with random digits}.{\BBCQ}
\newblock
\APACjournalVolNumPages{Applied Math Series}{12}{36-38}{5}.
\PrintBackRefs{\CurrentBib}

\end{thebibliography}

\end{document}